\documentclass[11pt,reqno,a4paper]{amsart}

\usepackage{amsmath,amsfonts,amsthm,mathrsfs,amssymb,cite}
\usepackage[usenames]{color}
\usepackage{graphicx,epstopdf,booktabs,array}
\usepackage{paralist,verbatim,caption,subcaption}

\usepackage{latexsym} 
\usepackage{enumerate}

\theoremstyle{plain}
\newtheorem{theorem}{Theorem}[section]
\newtheorem*{theorem*}{Theorem}
\newtheorem{lemma}[theorem]{Lemma}

\newtheorem{proposition}[theorem]{Proposition}
\newtheorem*{proposition*}{Proposition}

\newtheorem*{conjecture*}{Conjecture}

\theoremstyle{definition}
\newtheorem{definition}[theorem]{Definition}

\theoremstyle{remark}
\newtheorem{remark}[theorem]{Remark}

\DeclareMathOperator{\supp}{supp}

\newcommand{\abs}[1]{\left\lvert #1 \right\rvert}
\newcommand{\norm}[1]{\left\lVert #1 \right\rVert}

\newcommand{\R}{\mathbb{R}}
\newcommand{\C}{\mathbb{C}}

\renewcommand{\eqref}[1]{\textnormal{(\ref{#1})}}

\numberwithin{equation}{section}

\title[Recovery by a single far-field pattern]{Recovering piecewise constant refractive indices by a single far-field pattern}

\author{Emilia Bl{\aa}sten}

\address{Department of Mathematics, University of Helsinki, Helsinki, Finland}

\email{emilia.blasten@iki.fi}

\author{Hongyu Liu}
\address{Department of Mathematics, City University of Hong Kong, Kowloon, Hong Kong SAR, China.}
\email{hongyu.liuip@gmail.com; hongyliu@cityu.edu.hk}

\begin{document}

\begin{abstract}

We are concerned with the inverse scattering problem of recovering an
inhomogeneous medium by the associated acoustic wave measurement. We
prove that under certain assumptions, a single far-field pattern
determines the values of a perturbation to the refractive index on the
corners of its support. These assumptions are satisfied for example in
the low acoustic frequency regime. As a consequence if the
perturbation is piecewise constant with either a polyhedral nest
geometry or a known polyhedral cell geometry, such as a pixel or voxel
array, we establish the injectivity of the perturbation to far-field
map given a fixed incident wave. This is the first unique determinancy
result of its type in the literature, and all of the existing results
essentially make use of infinitely many measurements.


  
\end{abstract}

\maketitle

\section{Introduction}\label{sec:Intro}

\subsection{Mathematical setup}
Let $V\in L^\infty(\mathbb{R}^n)$, $n=2,3$, be a bounded measurable
complex-valued function. Let $\Omega$ be a bounded domain in $\R^n$
with a connected complement $\R^n\setminus\overline{\Omega}$. Assume
that $\supp(V)\subset\overline{\Omega}$. Physically speaking, $V$ is
the material parameter of an inhomogeneous acoustic medium supported
on $\Omega$, with $\Re V$ related to the refractive index and $\Im V$
related to the energy loss into the medium. In what follows, we simply
call $V$ the refractive index of the medium. We let $u^i$ be an entire
solution to the Helmholtz equation,
\begin{equation}\label{eq:Helm}
  (\Delta+k^2)u^i=0\quad\mbox{in}\ \ \ \mathbb{R}^n.
\end{equation} 
Consider the following scattering system for $u\in H_{loc}^1(\R^n)$,
\begin{equation}\label{eq:Helm1}
  \begin{cases}
    \big(\Delta + k^2(1+V)\big) u =
    0\quad\mbox{in}\ \ \R^n,\medskip\\ \displaystyle{
     \lim_{r:=|x|\rightarrow+\infty} \abs{x}^{\frac{n-1}{2}}
      \big(\partial_r - ik\big) (u-u^i) = 0},
  \end{cases}
\end{equation}
where $\partial_r$ is the derivative along the radial direction from
the origin.  The last limit in \eqref{eq:Helm1} is known as the
Sommerfeld radiation condition and it holds uniformly with respect to
the angular variable $\hat x:=x/|x|\in\mathbb{S}^{n-1}$ as
$r\to+\infty$. The radiation condition implies the existence of a
far-field pattern. More precisely there is a real-analytic function on
the unit-sphere at infinity $A_{u^i}: \mathbb{S}^{n-1}\to\C$ such that
\begin{equation}\label{eq:Helm3}
  u(r\hat x) = u^i(r\hat x) + \frac{e^{ikr}}{r^{(n-1)/2}} A_{u^i}(\hat
  x) + \mathcal{O} \Big( \frac{1}{r^{(n+1)/2}} \Big),
\end{equation}
which holds uniformly along the angular variable $\hat x$.
  
In the physical scenario, \eqref{eq:Helm1} describes time-harmonic
acoustic scattering due to the presence of an incident wave field
$u^i$ and an inhomogeneous medium supported on $\Omega$. The functions
$u$ and $u^s:=u-u^i$, respectively, signify the total and scattered
wave fields. $A_{u^i}$ can be measured by a physical apparatus, and it
encodes the information of the scattering medium $V$. An important
inverse problem arising in practical applications is to recover $V$
from the knowledge of $A_{u^i}(\hat x)$. This problem lies at the core
of many areas of science and technology including radar and sonar,
geophysical exploration and medical imaging. In this paper, we are
mainly concerned with the unique recovery issue of this inverse
problem. That is, given a certain set of measurement data, we shall
show what kind of unknowns one can recover; or in other words, given
the class of the unknowns, what kind of measurement data can ensure
the successful recovery. If one introduces an operator $F$, which
sends the unknown $V$ to the associated far-field pattern $A_{u^i}$:
\begin{equation}\label{eq:ip1}
  F(V)=A_{u^i},
\end{equation}
where $F$ is defined by the scattering system \eqref{eq:Helm1}, then
it can readily be shown that $F$ is nonlinear. Hence, the unique
recovery result can also ensure the global existence of a unique
solution to the nonlinear inverse problem \eqref{eq:ip1}. In this
paper, we are particularly interested in the case with minimal
measurement data: a single far-field pattern. By a single far-field
pattern, we mean that $A_{u^i}(\hat x)$ is given for all $\hat
x\in\mathbb{S}^{n-1}$ and a single fixed incident wave $u^i$. On the
other hand, we note that if $A_{u^i}(\hat x)$ is given for $\hat x$
from any open patch of $\mathbb{S}^{n-1}$, then by analytic
continuation, it is known for all $\hat x\in\mathbb{S}^{n-1}$.
  
The main result that we aim to establish is that if the medium $V$ is
piecewise constant within two general polyhedral geometries, then one
can uniquely recover it by a suitable single far-field measurement.

\subsection{Connection to existing studies and discussions}
For a general medium parameter $V\in L^\infty(\Omega)$, the most
widely known unique recovery results in the inverse scattering
community makes use of infinitely many far-field patterns; see
\cite{CK,Isa} for convenient references. Here, by infinitely many
far-field patterns, we mean $A_{u^i}$ are given corresponding to
incident plane waves $u^i=e^{ikx\cdot d}$ with all possible
$d\in\mathbb{S}^{n-1}$ and a fixed $k\in\mathbb{R}_+$. A crucial
ingredient is that the data set $\{A_{e^{ikx\cdot d}}(\hat x)\}_{(\hat
  x, d)\in\mathbb{S}^{n-1}\times\mathbb{S}^{n-1}}$ is equivalent to
the Cauchy data set, $\mathcal{C}_V:=\{(u|_{\partial\Omega},
\partial_\nu u|_{\partial\Omega}); \big(\Delta+k^2(1+V)\big)u = 0
\mbox{ in } \Omega\}$, where $\nu\in\mathbb{S}^{n-1}$ is the exterior
unit normal vector to $\partial\Omega$, cf. \cite{Nachman88,Uhl1}.
Hence, in the case with infinitely many measurement, the study of the
inverse scattering problem of recovering $V$ can be reduced to the
study of the inverse boundary value problem of recovering $V$ from the
associated Cauchy data set. The establishment of the unique recovery
result for the aforementioned inverse boundary value problem is mainly
based on the Sylvester-Uhlmann method pioneered in
\cite{Sylvester--Uhlmann}, and on the Bukhgeim method in two
dimensions \cite{Buk}. We also refer to a recent survey paper
\cite{Uhl0} for many subsequent relevant developments.

The problem of recovering a potential from the Cauchy data is formally
overdetermined in three and higher dimensions. This manifests in a
much tardier solution to the problem in the plane \cite{Buk} than in
higher dimensions \cite{Sylvester--Uhlmann}. In a similar vein, the
conjectured integrability limit of $L^{n/2}$ for potential recovery
\cite{Jerison--Kenig} has been achieved in the formally overdetermined
case \cite{Ferreira--Kenig--Salo}, but the two dimensional formally
determined case is lagging behind \cite{BTW17,BOY}. Also, the formally
determined \emph{backscattering problem} is still largely open, but
with recent progress on an admissible class of potentials
\cite{RU,RU2}. In two recent articles by one of the authors
\cite{HLL,LiuLiu}, it is shown that if the medium $V$ is from a
certain special class and contains an impenetrable obstacle, then
$\{A_{e^{ikx\cdot d}}(\hat x)\}_{(\hat x,
  k)\in\mathbb{S}^{n-1}\times\mathbb{R}_+}$ with a fixed
$d\in\mathbb{S}^{n-1}$ and $k$ in an interval is enough measurement
data. This problem is formally determined. The problem of our paper
--- the unique determination of a potential by a single far-field
measurement --- is formally underdetermined, so stands no chance of
being completely solved. However by restricting the potentials into a
space that's still useful with respect to applications, but has fewer
degrees of freedom, this problem is avoided.

Making use of the corner scattering results in \cite{BPS,BLLW,BL2017,EH1}, the
authors in \cite{BL2016,HSV} showed that if a medium $V$ is supported on
a convex polyhedral domain $\Omega$, then $\Omega$ can be uniquely
determined (in a stable way) by a single far-field pattern. Earlier work on \emph{the
  scattering support} also recover useful information from a single
far-field pattern \cite{Kusiak--Sylvester}. In this article we show
that not only the domain, but also the values of the potential can be
uniquely determined at the vertices by a far-field pattern created by
a suitable incident wave. As a consequence, we are able to establish a
much more general result in the same class of measurements:
determining a polyhedral piecewise-constant medium parameter by a
single far-field pattern.

If the medium is piecewise constant on a suitable known polyhedral
grid, for example a pixel or voxel array, then one can uniquely
recover it by a single far-field pattern. If the grid is not known
a-priori, then we can recover the potential assuming that it is
piecewise constant on a \emph{nested polyhedral grid}. To the best of
our knowledge, this is the first uniqueness result of its type in the
literature.

There is one important implication of our uniqueness result,
particularly from the numerical point of view. The finite element
method (FEM) is widely used in the engineering community. Using the FEM approximation, the medium
parameter $V$ is usually approximated by a piecewise polynomial
function with a polyhedral triangulation of its support. That is,
\begin{equation}\label{eq:fem1}
V\approx V_h:=\sum_{\Sigma_j\in\mathcal{T}_h} P_j \chi_{\Sigma_j},
\end{equation}
where $\mathcal{T}_h$ is a polyhedral triangulation of $\Omega$ and $P_j$ is a polynomial function supported in $\Sigma_j$, a.k.a 
a finite element basis function. Here, $h\in\mathbb{R}_+$ signifies the mesh size of the FEM approximation. Hence, under the approximation \eqref{eq:fem1}, instead of solving \eqref{eq:ip1}, one actually solves
\begin{equation}\label{eq:ip1app}
F(V_h)=A_{u^i}(V_h)+\epsilon_h,\quad \epsilon_h:=A_{u^i}(V)-A_{u^i}(V_h). 
\end{equation}
Clearly, by
using our uniqueness result, if the approximation is made with a
piecewise constant function associated to a certain polyhedral
triangulation as discussed in this work, there is a one-to-one correspondence between $V_h$ and $A_{u^i}(V_h)$ through the far-field map $F$, which shall be of critical importance in 
analyzing the convergence of the FEM approximation \eqref{eq:ip1app} . 
Hence, the theoretical results obtained in the current article might
motivate some novel numerical and practical applications, which are definitely worth further investigation.

We are also led to the following conjecture in inverse scattering
theory: with one or a few incident plane waves, if two mediums $V$ and
$V'$ of certain type are different, then the associated far-field
patterns should be distinguishable from each other. At least,
according to the current article, the conjecture is true for piecewise
constant mediums supported in certain polyhedral domains. However it
is known to be false for some radially symmetric mediums
\cite{CM,CPS}. We believe that the conjecture would hold for a general
class of mediums, say, with piecewise polynomial material
parameters. We shall investigate this and other interesting issues in
our future work.

\subsection{Improvements to the method}
Finally, we would like to comment on the mathematical argument in
deriving our uniqueness results. In addition to the completely new
potential determination theorems, we improve past corner scattering
methods. The first step is to show that not only can the vertices of
the scatterer be recovered, but also the value of the potential on
these points. This is an extremely useful improvement over past
determination results \cite{BL2016,HSV,EH1}. For this we will have to use
the total wave of one of the potentials as an incident wave for the
other one. Technically we have to use the Taylor expansion of such a
non-smooth wave. This difficulty is avoided by making sure that the
total wave does not vanish at vertices.

A significant technical improvement to the corner scattering method is
the following: our results will now work for arbitrary convex
polyhedrons in 3D, hence improving the recovery results of
\cite{BL2016,HSV} where the three dimensional objects had to be
rectangular boxes. This generalization is made possible by studying
more carefully the Laplace transforms of polyhedral cones, as inspired
by \cite{ABR}, and works when the total wave is nonzero at the
corner. See Lemma \ref{Laplace}. We leave for future work the
technically much more challenging problem of dealing with total waves
that vanish to arbitrarily high order at some of the vertices of the
scatterers. 

Once the corners and values at the corners are recovered, this means
that we know the piecewise constant potential in a neighbourhood of
its corners. To progress further we need to propagate the equality of
the total waves of the two potentials. This is done by an application
of Holmgren's uniqueness theorem. As an aside, what we prove implies a
variation of Holmgren's theorem in the setting of two equations,
$(\Delta+q)u=0$ and $(\Delta+q')u'=0$: if we know that $u=u'$ and
$\partial_\nu u = \partial_\nu u'$ on the boundary near a corner, then
under the assumptions of our theorems we have $u=u$ inside, and also
$q=q'$ at the corner.

We state three theorems in this paper. One is the unique recovery of
shape and values at the vertices of a polyhedral scatterer. The second
one is the complete recovery of a piecewise constant potential defined
on an a-priori known grid of polyhedral cells. Think of pixels for
example. The last theorem shows a case where even the cell structure
is not known beforehand, however we must have the knowledge that the
potentials have a nested geometry. These are not the ultimate limits
of this method, however it is still unclear how far can one go.

The rest of the paper is organized as follows. In Section 2, we
present the main uniqueness results. Section 3 is devoted to the
proofs.

\section{Main results}

\subsection{Geometric setup}
We shall consider the inverse problem \eqref{eq:ip1} with a piecewise
constant refractive index of the following form
\begin{equation}\label{eq:pw1}
  V=\sum_j V_j \chi_{\Sigma_j}, \quad\overline{\Omega} = \bigcup_j
  \overline{\Sigma_j},
\end{equation}
where $V_j\in\C$ are constants, and $\Sigma_j$ are mutually disjoint
bounded open subsets in $\mathbb{R}^n$. For completeness we also
recover the corner value of potentials that are H\"older-continuous
near the vertices on a polygonal domain.

In the piecewise constant case, we shall consider two geometric setups
of the subdomains $\Sigma_j$ that can be described as follows. The
first one is referred to as the polyhedral cell geometry, see
Figure~\ref{fig:geo}.(A) for a schematic illustration.
\begin{definition}
  An \emph{admissible cell} $P\subset\R^n$ is a bounded open convex
  polytope, i.e. a polygon in 2D and a polyhedron in 3D.
\end{definition}
\begin{definition} \label{cellDef}
  For $j\in\mathbb Z_+$ let each of $\Sigma_j\subset\R^n$ be an
  admissible cell or the empty set, $\cup_j \overline{\Sigma_j}$ is
  simply connected, bounded and $\Sigma_j \cap \Sigma_k = \emptyset$
  if $j\neq k$. A bounded potential $V \in L^\infty(\R^n)$ is said to
  be \emph{piecewise constant with polyhedral cell geometry} if
  \begin{enumerate}
    \item there are constants $V_j\in\C$ such that
      \[
      V(x) = \sum_{j=1}^\infty V_j \chi_{\Sigma_j}(x),
      \]
      with $V_j=0$ if $\Sigma_j=\emptyset$, and
    \item each $\Sigma_j\neq\emptyset$ has a vertex that can be
      connected to infinity by a path that stays distance $d\geq
      d_0>0$ from any $\Sigma_k$ with $k>j$.
  \end{enumerate}
\end{definition}
The latter condition is satisfied for example by taking any finite
number of unit squares or cubes that are part of a lattice, and
ordering them from left to right and top to bottom. In essence if we
approximate a potential's graph by a discrete picture consisting of
pixels or voxels, then this approximation represents a piecewise
constant potential with polyhedral cell geometry.

\begin{figure}[t]
  \centering
  \begin{subfigure}[b]{0.37\textwidth}
    \includegraphics{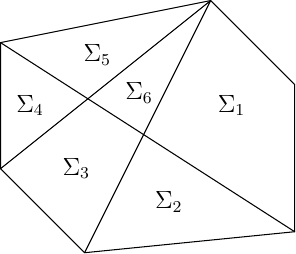}
    \caption{}
  \end{subfigure}
  \qquad\qquad
  \begin{subfigure}[b]{0.37\textwidth}
    \includegraphics{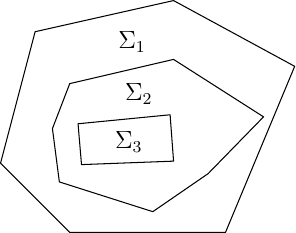}
    \caption{}
  \end{subfigure} 
  \caption{Schematic illustration of the two polyhedral geometries in
    two dimensions. (A) Polyhedral cell geometry; (B) Polyhedral
    nested geometry.}
  \label{fig:geo}
\end{figure}

The second one is referred to as the polyhedral nested geometry.  See
Figure~\ref{fig:geo}.(B) for a schematic illustration.
\begin{definition} \label{nestedDef}
  For $j\in\mathbb Z_+$ let each of $D_j\subset\R^n$ be an admissible
  cell or the empty set, and
  \[
  D_j \Supset D_{j+1}.
  \]
  A bounded potential $V \in L^\infty(\R^n)$ is said to be
  \emph{piecewise constant with polyhedral nested geometry} if there
  are constants $V_j\in\C$, with $V_1\neq0$, $V_{j+1}\neq V_j$ such
  that
  \[
  V(x) = \sum_{j=1}^\infty V_j \chi_{\Sigma_j}(x)
  \]
  where $\Sigma_j = D_j \setminus \overline{D_{j+1}}$.
\end{definition}

Finally we define potentials analogous to ones used in previous corner
scattering results \cite{PSV,HSV}.

\begin{definition} \label{nonConstantDef}
  A potential $V \in L^\infty$ is \emph{a non-constant admissible
    potential} if there is an admissible cell $P\subset\R^n$ and
  bounded function $\varphi\in L^\infty(\R^n)$ such that $V = \chi_P
  \varphi$. Moreover we require that $\varphi$ be H\"older
  $C^\alpha$-continuous in a neighbourhood of each of the vertices of
  $P$ with $\alpha>0$ in 2D and $\alpha>1/4$ in 3D. Finally, the
  function $\varphi$ must not vanish at any of the vertices.
\end{definition}

\subsection{Unique recovery results} 

Our results depend on the total field $u$ not vanishing at the
vertices of the various polytopes $\Sigma_j$. The nodal (or vanishing)
set cannot be too large in general. For concreteness Lemma
\ref{emptyNodal} gives a sufficient condition: the nodal set is empty
for low enough frequencies with incident plane-waves; see Remark~\ref{rem:new1} for more relevant discussion.

\begin{theorem} \label{nonConstantThm}
  Let $n\in\{2,3\}$, $k>0$ and $V=\chi_P\varphi, V'=\chi_{P'}\varphi'$
  be two non-constant admissible potentials.

  Let $u^i$ be an incident wave such that $u(x_c)\neq0$ or
  $u'(x_c)\neq0$ for the total waves $u,u'$ at each vertex $x_c$ of
  $P$ or $P'$. Assume that
  \[
  A_{u^i} = A'_{u^i}
  \]
  for the far-field patterns arising from $V$ and $V'$, respectively.
  Then $P=P'$ and $\varphi(x_c)=\varphi'(x_c)$ on each vertex $x_c$ of
  $P=P'$.
\end{theorem}

\begin{theorem} \label{cellThm}
  Let $n\in\{2,3\}$ and $k>0$. Let $V$ and $V'$ be two piecewise
  constant potentials with common polyhedral cell geometry.

  Let $u^i$ be an incident wave such that $u(x_c)\neq0$ or
  $u'(x_c)\neq0$ for each vertex $x_c$ of the cells of $V$ and
  $V'$. If $A_{u^i} = A'_{u^i}$ then $V=V'$.
\end{theorem}

\begin{theorem} \label{nestedThm}
  Let $n\in\{2,3\}$ and $k>0$. Let $V$ and $V'$ be two piecewise
  constant potentials with polyhedral nested geometry.

  Let $u^i$ be an incident wave such that $u(x_c)\neq0$ or
  $u'(x_c)\neq0$ for each vertex $x_c$ of the cells of $V$ and
  $V'$. If $A_{u^i} = A'_{u^i}$ then $V=V'$.
\end{theorem}

\section{Proofs}

We first present a useful lemma.
\begin{lemma} \label{scatteredNorm}
  Let $u\in H_{loc}^1(\mathbb{R}^n)$ be the solution to
  \eqref{eq:Helm1}. Suppose that $\Omega\subset B_R$, where $B_R$ is a
  central ball of radius $R\in\mathbb{R}_+$.  Then $u^s\in H^2(B_R)$
  and there exists $C_0\in\mathbb{R}_+$ such that when $k^2
  \norm{V}_\infty \leq C_0$,
  \begin{equation}\label{eq:est1}
    \norm{u^s}_{H^2(B_R)}\leq C k^2 \norm{V}_{L^\infty(\Omega)}
    \norm{u^i}_{L^2(\Omega)},
  \end{equation}
  where $C$ is a positive constant depending only on $R$.
\end{lemma}
\begin{proof}
  We only prove the case with $n=3$, and the other cases can be proved
  by following a similar argument.  Let
  \[
  \Phi(x)=\frac{1}{4\pi}\frac{e^{ik|x|}}{|x|},\ \ x\in\mathbb{R}^3\ \ |x|\neq
  0.
  \]
  and define
  \[
  \mathcal{L}_V(u)=\int_{\Omega} \Phi(x-y) V(y) u(y)\ dy.
  \]
  We know there is the following integral relation (cf. \cite{CK})
  \begin{equation}\label{eq:ls}
    u^s(x)=k^2\mathcal{L}_V(u^s)(x)+k^2\mathcal{L}_V(u^i)(x),\quad
    x\in\mathbb{R}^3.
  \end{equation}
  Using the facts that $u^s\in H^1(B_R)$, and $\mathcal{L}_V$ maps
  $L^2(B_R)$ continuously into $H^2(B_R)$ with norm $k^2
  \norm{V}_\infty$ (cf. Theorem 8.2 in \cite{CK}), one can easily see
  that $u^s\in H^2(B_R)$. For $k$ or $\norm{V}_\infty$ sufficiently
  small, we have from \eqref{eq:ls} that $u^s =
  (I-k^2\mathcal{L}_V)^{-1} (k^2\mathcal{L}_V(u^i))$ and hence
  \begin{equation}\label{eq:ls2}
    \|u^s\|_{L^2(B_R)}\leq C k^2 \norm{V}_{L^\infty(\Omega)}
    \norm{u^i}_{L^2(\Omega)},
  \end{equation}
  where $C$ is positive constant depending only on $R$. Finally, by
  applying the estimate in \eqref{eq:ls2} to the RHS of \eqref{eq:ls},
  together with the use of the mapping property of $\mathcal{L}_V$
  again, one can show \eqref{eq:est1}. The proof is complete.
\end{proof}

\begin{lemma} \label{emptyNodal}
  Let $n\in\{2,3\}$, $k>0$ and let $B_R$ be a central ball of radius
  $R\in\R_+$. There is $C>0$ such that if $V\in L^\infty(B_R)$ with
  \begin{equation}\label{eq:new1}
  k^2 \norm{V}_\infty < C
  \end{equation}
  and $u\in H^2(B_R)$ is the total field created by $V$ and an
  incident unit-modulus plane-wave $u^i$ of wave-number $k$, then $u$
  cannot vanish in $B_R$.
\end{lemma}

\begin{proof}
  Lemma~\ref{scatteredNorm} and the Sobolev embedding of $H^2$ into
  $L^\infty$ in 2D and 3D imply that $\norm{u^s}_\infty<1$ in $B_R$
  then. Hence
  \[
  \abs{u(x)} \geq \abs{u^i(x)} - \abs{u^s(x)} > 0.
  \]
  for any $x\in B_R$.
\end{proof}

\begin{remark}\label{rem:new1}
It is remarked that Lemma~\ref{emptyNodal} provides a sufficient scenario for the non-vanishing of the total wave fields at the vertex points, which is required in 
Theorems~\ref{nonConstantThm} and \ref{cellThm}. According to \eqref{eq:new1}, if the wave-number is sufficiently small compared to the refractive index, then the non-vanishing requirement 
is fulfilled. This is the only place where we make use of the low wave-number assumption. In our subsequent analysis, say Lemmas~\ref{Laplace} and \ref{boundary2corner}, the CGO (complex geometric optics) argument does not depend on the wave-number $k$ as long as it is fixed. Moreover, we would like to point out that the non-vanishing of the total wave fields may hold in more general scenarios without this low wave-number assumption. 
\end{remark}

The following lemma is a key identity in corner scattering. It is a
slight modification of the Alessandrini identity that is used in
inverse problems in general. It becomes very useful when $V$ and $V'$
have their supports on a common cone, and $u=u'$ outside that
cone. Letting $u_0$ be a complex geometrical optics (CGO) solution is a good
choice then.
\begin{lemma} \label{orthogonality}
  Let $U\subset\R^n$ be a bounded Lipschitz domain and $q,q'\in
  L^\infty(U)$. If $u_0, u, u' \in H^2(U)$ satisfy
  \begin{align*}
    (\Delta+q)u_0 = 0\\ (\Delta+q)u = 0\\ (\Delta+q')u' = 0
  \end{align*}
  in $U$ then
  \[
  \int_U (q-q') u' u_0 dx = \int_{\partial U} \big( u_0 \partial_\nu
  (u-u') - (u-u') \partial_\nu u_0 \big) d\sigma.
  \]
\end{lemma}
\begin{proof}
  We have $(\Delta+q)(u'-u) = (q-q')u'$. By Green's identities
  \begin{align*}
    &\int_U (q-q') u' u_0 dx = \int_U u_0 (\Delta+q)(u'-u) dx
    \\ &\qquad = \int_{\partial U} \big( u_0 \partial_\nu (u-u') -
    (u-u') \partial_\nu u_0\big) d\sigma + \int_U (u'-u) (\Delta+q)u_0
    dx
  \end{align*}
  and the claim follows because $(\Delta+q)u_0=0$.
\end{proof}

Another key-element in the proofs of corner scattering is to show that
the Laplace transforms of the indicator functions of certain cones do
not vanish when evaluated at the complex geometrical optics parameter
$\rho\in\C^n$, $\rho\cdot\rho=0$. Here we generalize previous results
of \cite{BL2016} to allow for arbitrary convex polyhedrons in 3D. All
previous papers on the topic, excluding \cite{EH1} had required that
the three dimensional polyhedra were rectangular boxes. The following
technique of splitting the Laplace transform into a simple but large
integral, and a difficult but bounded one was inspired by \cite{ABR}.

\begin{lemma} \label{Laplace}
  Let $n\in\{2,3\}$ and $\mathcal{C}\subset\R^n$ be an open convex
  polyhedral cone smaller than a half-space. Then there is
  $\rho\in\C^n$, $\rho\cdot\rho=0$, $\rho\neq0$ such that
  \[
  \int_{\mathcal C} e^{\rho\cdot x} dx \neq 0
  \]
  and the integral converges.
\end{lemma}

\begin{proof}
  We may assume that the vertex of $\mathcal{C}$ is the origin. If
  $\mathcal{C}$ is a simplex cone, i.e. it is generated by $n$
  linearly independent unit vectors
  \[
  w_1, \ldots, w_n \in \mathbb S^{n-1}, \qquad \mathcal{C} = \{
  \alpha_1 w_1 + \ldots + \alpha_n w_n \mid \alpha_j > 0\},
  \]
  then by a linear change of variables we see that
  \[
  \int_\mathcal{C} e^{\rho\cdot x} dx = \frac{\abs{w_1 \wedge \ldots
      \wedge w_n}}{(-\rho\cdot w_1) \cdot \ldots \cdot (-\rho\cdot
    w_n)}
  \]
  where $\abs{w_1 \wedge \ldots \wedge w_n}$ is the determinant of the
  $n\times n$ matrix with column vectors $w_1,\ldots,w_n$. The two
  dimensional claim follows immediately by letting $-\Re \rho$ be in
  the interior of $\mathcal{C}$.

  For the three dimensional case let the generators of $\mathcal{C}$
  be
  \[
  w_1, \ldots, w_m
  \]
  in other words $\{t w_j \mid t>0\}$ are its edges. We may order
  these vectors so that the segments $\{tw_j+(1-t)w_{j+1}\}$ and
  $\{tw_m+(1-t)w_1\}$ are on $\partial \mathcal{C}$. Fix $w=w_1$ and
  let $z$ be a unit vector such that
  \[
  z\cdot w=0, \qquad z\cdot w_j>0
  \]
  for $j>1$. For $\varepsilon>0$ let
  \[
  R_\varepsilon = \frac{z+\varepsilon w}{\abs{z+\varepsilon w}}
  \]
  and let $I$ be a unit vector such that $I\cdot z=I\cdot
  w=0$. Finally let
  \[
  \rho = -R_\varepsilon - i I
  \]
  and we observe that $\rho\cdot\rho=0$, $\rho\neq0$. Moreover we note
  that the Laplace transform converges when $\varepsilon$ is small
  enough but positive.

  Let $\mathcal{C}_1$ be the open cone generated by $w=w_1$, $w_2$ and
  $w_m$. We have
  \begin{align*}
    &\rho\cdot w_1 = \frac{-\varepsilon}{\abs{z+\varepsilon w}}
    \longrightarrow 0\\ &\rho\cdot w_j = -\frac{z\cdot w_j +
      \varepsilon w\cdot w_j}{\abs{z+\varepsilon w}} - i I\cdot w_j
    \longrightarrow -z\cdot w_j-i I\cdot w_j
  \end{align*}
  for $j\in\{2,m\}$ as $\varepsilon\to0$. Hence $\abs{\rho\cdot w_j}
  \geq \abs{z\cdot w_j} > 0$. Thus by the formula for the transform of
  a simplex cone
  \[
  \lim_{\varepsilon\to0} \abs{\int_{\mathcal{C}_1} e^{\rho\cdot x} dx}
  = \infty.
  \]

  We shall prove that the integral over
  $\mathcal{C}'=\mathcal{C}\setminus\overline{\mathcal{C}_1}$ is
  bounded next. This cone is generated by the vectors
  $w_2,\ldots,w_m$. There is a positive constant
  $\delta(w_2,\ldots,w_m)$ such that if $\theta\in \mathcal{C}'$ is a
  unit vector then
  \[
  \theta = \alpha_2w_2 + \ldots + \alpha_mw_m, \qquad \max_j \alpha_j
  \geq \delta(w_2,\ldots,w_m) > 0
  \]
  for some non-negative real numbers $\alpha_j$. Then
  \begin{align*}
    &\lim_{\varepsilon\to0} \Re\rho\cdot\theta = - \sum_{j=2}^m
    \alpha_j z\cdot w_j \leq - \min_{j\geq2} z\cdot w_j \sum_{j=2}^m
    \alpha_j\\ &\qquad\leq - \min_{j\geq2} z\cdot w_j \,
    \delta(w_2,\ldots,w_m) < 0
  \end{align*}
  by the choice of $z$. By the triangle inequality
  \[
  \abs{\int_{\mathcal{C}'} e^{\rho\cdot x} dx} \leq \int_{\mathcal{C}'
    \cap \mathbb S^2} \int_0^\infty e^{\Re\rho\cdot\theta r} r^2 dr
  d\sigma(\theta)
  \]
  and hence, writing $\delta(z,w_2,\ldots,w_m) = \min_{j\geq2} z\cdot
  w_j \, \delta(w_2,\ldots,w_m)$, we get
  \[
  \lim_{\varepsilon\to0} \abs{\int_{\mathcal{C}'} e^{\rho\cdot x} dx}
  \leq \sigma(\mathcal{C}'\cap\mathbb S^2) \int_0^\infty
  e^{-\delta(z,w_2,\ldots,w_m)r} r^2 dr = C(z,w_2,\ldots,w_m)
  \]
  the latter of which is a finite constant.

  The integral over $\mathcal{C}_1$ can be made arbitrarily large
  while the integral over
  $\mathcal{C}\setminus\overline{\mathcal{C}_1}$ stays uniformly
  bounded. Hence the whole integral can be made arbitrarily large, and
  thus nonzero.
\end{proof}

\begin{lemma} \label{boundary2corner}
  Let $\Omega\subset\R^n$, $n\in\{2,3\}$ be a bounded domain. Let
  $q,q'\in L^\infty(\Omega)$ and let $u, u' \in H^2(\Omega)$ solve
  \[
  (\Delta+q)u = 0, \qquad (\Delta+q')u' = 0
  \]
  in $\Omega$. Let $x_c \in \partial\Omega$ be a point for which there
  is a neighbourhood $B$ and an admissible cell $\Sigma$ having $x_c$
  as vertex such that $B\cap\Omega=B\cap\Sigma$.

  Assume that $u=u'$ and $\partial_\nu u = \partial_\nu u'$ on
  $B\cap\partial\Omega$. If $q$ and $q'$ are $C^\alpha$ uniformly
  H\"older-continuous functions with $\alpha>0$ in 2D and $\alpha>1/4$
  in 3D in $B \cap \Omega$ then
  \[
  (q-q')(x_c) u(x_c) = (q-q')(x_c) u'(x_c) =0.
  \]
\end{lemma}

\begin{proof}
  Choose coordinates such that $x_c=\bar0$. We may assume that $B$ is
  a small $x_c$-centred disc with $B\cap\Omega=B\cap\mathcal{C}$ for
  some open cone $\mathcal{C}\subset\R^n$ with vertex $x_c$. We shall
  prove the case $(q-q')(x_c)u'(x_c)=0$. The other one follows by
  symmetry.

  For any $\rho \in \C^n$ with $\rho\cdot\rho=0$ and $\abs{\Re\rho}>0$
  large enough depending on $q, B, \mathcal C$, the equation
  \[
  \big(\Delta+q\big) u_0 = 0
  \]
  has a complex geometrical optics solution $u_0(x) = \exp(\rho\cdot
  x)(1+\psi(x))$ in $B\cap\mathcal C$. This is given for example by
  Theorem 3.1 in \cite{PSV} for 2D, Proposition 7.6 in \cite{BL2016}
  or Lemmas 2.1 and 3.2 in \cite{HSV}, all of them requiring
  $\alpha>1/4$ in 3D. They imply the existence of $p\geq2$ and
  $\delta>0$, both independent of $\rho$, such that
  \[
  \norm{\psi}_{L^p(B\cap\mathcal C)} \leq C
  \abs{\Re\rho}^{-n/p-\delta}
  \]
  and $u_0\in H^2(B\cap\mathcal C)$.

  Recall that $u=u'$ and $\partial_\nu u = \partial_\nu u'$ on
  $B\cap\partial\mathcal C$. Lemma~\ref{orthogonality} gives
  \[
  \int_{B\cap\mathcal C} (q-q') u' u_0 dx = \int_{\mathcal C \cap
    \partial B} \big( u_0 \partial_\nu (u-u') - (u-u') \partial_\nu
  u_0 \big) d\sigma.
  \]
  Choose $\rho$ such that $\exp(\Re\rho\cdot x)$ decays exponentially
  in the cone $\mathcal C$. Then the boundary integral's absolute
  value can be estimated as
  \begin{align*}
    &\abs{\ldots} \leq \norm{u-u'}_{H^2(B\cap\mathcal C)}
    \norm{u_0}_{H^2(B\cap\mathcal C)} \\ &\qquad\leq C
    \norm{u-u'}_{H^2(B\cap\mathcal C)} (1+\abs{\Re\rho})
    e^{-c\abs{\Re\rho}} (1+\norm{\psi}_{H^2(B\cap\mathcal C)})
    \\ &\qquad\leq C e^{-c'\abs{\Re\rho}}
  \end{align*}
  when $\abs{\Re\rho}$ is large enough: $\norm{\psi}_{H^2}$ is of the
  order $\abs{\Re\rho}^2$ and by Lemma~\ref{scatteredNorm} the $H^2$
  norm of $u-u'=u^s-{u'}^s$ is finite. Let us split and estimate the
  integral over $B\cap\mathcal C$ next. Write $\delta_q = q-q'$ and
  split
  \begin{align*}
    (q-q')(x) &= \delta_q(\bar0) + \big( \delta_q(x) - \delta_q(\bar0)
    \big),\\ u'(x) &= u'(\bar0) + \big( u'(x) - u'(\bar0)
    \big),\\ u_0(x) &= e^{\rho\cdot x} + e^{\rho\cdot x}\psi(x),
  \end{align*}
  where the pointwise values of $u'$ are well-defined since
  $H^2$-functions are continuous in 2D and 3D. We have the estimates
  \begin{align*}
    \abs{\delta_q(x) - \delta_q(\bar0)} &\leq (\norm{q}_{C^\alpha} +
    \norm{q'}_{C^\alpha}) \abs{x}^\alpha, \\ \abs{u'(x) - u'(\bar0)}
    &\leq C_B \norm{u'}_{H^2(B)} \abs{x}^{1/2},
    \\ \norm{\psi}_{L^p(B)} &\leq C \abs{\Re\rho}^{-n/p-\delta}
  \end{align*}
  by the definition of H\"older-continuity, the Sobolev embedding of
  $H^2(B)$ into the space $C^{1/2}$ of H\"older-continuous functions
  in 2D and 3D, and the previous paragraph about the complex
  geometrical optics solution.

  If $\rho$ is such that the integral $\int_{\mathcal C}
  \exp(\rho\cdot x) dx$ converges, we see the following telescope
  identity:
  \begin{align*}
    &(q-q')(\bar0) u'(\bar0) \int_{\mathcal C} e^{\rho\cdot x} dx =
    (q-q')(\bar0) u'(\bar0) \int_{\mathcal C \setminus B} e^{\rho\cdot
      x} dx \\ &\qquad - u'(\bar0) \int_{B\cap\mathcal C} e^{\rho\cdot
      x} \big( \delta_q(x) - \delta_q(\bar0) \big) dx \\ &\qquad -
    \int_{B\cap\mathcal C} e^{\rho\cdot x} (q-q')(x) \big( u'(x) -
    u'(\bar0) \big) dx \\ &\qquad - \int_{B\cap\mathcal C}
    e^{\rho\cdot x} (q-q')(x) u'(x) \psi(x) dx \\ &\qquad +
    \int_{\mathcal C \cap \partial B} \big( u_0 \partial_\nu (u-u') -
    (u-u') \partial_\nu u_0 \big) d\sigma.
  \end{align*}
  We will estimate the various terms in the identity above next. The
  first and last term decay exponentially as $\abs{\Re\rho}\to\infty$
  since $\Re\rho\cdot x \leq - c \abs{\Re\rho}\abs{x}$ for
  $x\in\mathcal C$ is satisfied when $\int_{\mathcal C} \exp(\rho\cdot
  x)dx$ is finite. For the three other terms note that
  \[
  \abs{\int_{\mathcal C} e^{\rho\cdot x} \abs{x}^s dx} \leq C
  \int_0^\infty e^{-c\abs{\Re\rho}r} r^{s+n-1} dr =
  \abs{\Re\rho}^{-n-s} \int_0^\infty e^{-cr'} {r'}^{n+s-1} dr'
  \]
  and similarly
  \[
  \norm{e^{\rho\cdot x}}_{L^{p'}(\mathcal C)} \leq \left(
  \int_{\mathcal C} e^{-c \abs{\Re\rho}\abs{x} p'} dx \right)^{1/p'} =
  C \abs{\Re\rho}^{-n/p'}.
  \]
  Using the triangle and H\"older's inequalities with exponent
  $(p,{p'})$ where $1/p+1/{p'}=1$ and the estimates for $\delta_q$ and
  $u'$ from the previous paragraph, we see that
  \begin{align*}
    &\abs{ (q-q')(\bar0) u'(\bar0) \int_{\mathcal C} e^{\rho\cdot x}
      dx } \\ &\qquad \leq C \big( e^{-c \abs{\Re\rho}} +
    \abs{\Re\rho}^{-n-\alpha} + \abs{\Re\rho}^{-n-1/2} +
    \abs{\Re\rho}^{-n-\delta} \big)
  \end{align*}
  for $\abs{\Re\rho}$ large engouh.

  For the lower bound note that no matter how large we require
  $\abs{\Re\rho}$ to be, Lemma~\ref{Laplace} and a simple
  multiplication by a real number give the existence of $\rho$
  satisfying all the requirements stated previously, and that
  \[
  \abs{ \int_{\mathcal C} e^{\rho\cdot x} dx } \geq C
  \abs{\Re\rho}^{-n}
  \]
  for some $C\neq0$. This implies that $(q-q')(\bar0) u'(\bar0) = 0$.
\end{proof}

\begin{proposition}\label{equalCells}
  Let $D\subset\R^n$, $n\in\{2,3\}$, be a bounded domain and
  $P,P'\Subset D$ admissible cells. Let $\alpha>0$ in 2D and
  $\alpha>1/4$ in 3D. Let $q,q'\in L^\infty(D)$ be uniformly
  $C^\alpha$ H\"older continuous outside $P,P'$, respectively and
  $q=q'$ outside $P\cup P'$. Assume that $q$ restricted to $P$ is
  uniformly $C^\alpha$ H\"older continuous near each vertex $x_c$ of
  $P$, but that
  \[
  \lim_{\substack{x\in P\\ x\to x_c}} q(x) \neq \lim_{\substack{x\in
      D\setminus\overline{P}\\ x\to x_c}} q(x)
  \]
  and similarly for $q'$ in $P'$.

  Let $u, u' \in H^2(D)$ solve
  \[
  (\Delta+q)u = 0, \qquad (\Delta+q')u' = 0
  \]
  and $u=u'$ outside $P\cup P'$. Assume that $u(x_c)\neq0$ or
  $u'(x_c)\neq0$ at every vertex $x_c$ of $P\cup P'$. Then $P=P'$ and
  \[
  \lim_{\substack{x\in P\\ x\to x_c}} q(x) = \lim_{\substack{x\in
      P\\ x\to x_c}} q'(x).
  \]
\end{proposition}
\begin{proof}
  If $P=\emptyset=P'$ we are done. Otherwise use
  Lemma~\ref{boundary2corner} with $\Omega=P\cup P'$, $q,q',u,u'$ all
  restricted to $\Omega$, and $x_c$ a vertex of $P$ (for example) that
  does not belong to $\overline{P'}$. Since $u=u'$ in
  $D\setminus\overline{\Omega}$ we have
  \[
  (q-q')(x_c) u(x_c) = (q-q')(x_c) u'(x_c) = 0
  \]
  where $q(x_c)$ and $q'(x_c)$ are the limits of $q(x)$, $q'(x)$ as
  $x\to x_c$ in $\Omega$, and hence in $P$. Since $u(x_c)\neq0$ or
  $u'(x_c)\neq0$ we get $q(x_c)=q'(x_c)$. However $q'$ is continuous
  in a neighbourhood of $x_c$ in $D$ because this point is away from
  $P'$. Moreover $q'=q$ outside $\Omega$. Thus
  \[
  \lim_{\substack{x\in P\\ x\to x_c}} q(x) = q(x_c) = q'(x_c) =
  \lim_{\substack{x\in D\\ x\to x_c}} q'(x) = \lim_{\substack{x\in
      D\setminus\overline{P}\\ x\to x_c}} q'(x) = \lim_{\substack{x\in
      D\setminus\overline{P}\\ x\to x_c}} q(x)
  \]
  which is a contradiction.

  Hence all vertices of $P$ belong to $\overline{P'}$. Similarly we
  see that all vertices of $P'$ belong to $\overline{P}$. Admissible
  cells are the interior of the convex hull of their vertices, so we
  have $P \subset P' \subset P$ i.e. $P=P'$. However we may then use
  Lemma~\ref{boundary2corner} directly to see that $q(x_c)=q'(x_c)$
  because $q'$ is not continuous anymore around $x_c$ in $D$.
\end{proof}

For the recovery of the piecewise constant potentials we will also
need a tool to propagate the equality of the total fields $u=u'$ into
cells where we have shown that the potentials are equal. This tool is
Holmgren's uniqueness theorem.

\begin{lemma} \label{holmgren}
  Let $U\subset\R^n$ be a connected Lipschitz domain. Let $\Gamma
  \subset \partial U$ be a non-empty relatively open subset of the
  boundary of $U$. Let $w \in H^2(U)$ satisfy $(\Delta+c)w=0$ in $U$
  for some constant $c\in\C$ and assume that the Dirichlet and Neumann
  data of $w$ vanishes on $\Gamma$. Then $w=0$ in $U$.
\end{lemma}
\begin{proof}
  Let $B \subset \R^n$ be a smooth open neighbourhood of a boundary
  point in $\Gamma$ such that $B \cap \partial U \subset \Gamma$. Let
  $w^*$ be the extension of $w$ by zero to $B$. It is a
  distribution. We will show that $(\Delta+c)w^*=0$ in $B$. Let
  $\varphi \in C^\infty_0(B)$. Then by noting that $(\Delta+c)w=0$ in
  $U$, and by the vanishing of the Cauchy data of $w$ on $\Gamma$, and
  the vanishing of the Cauchy data of $\varphi$ on $\partial B$, there
  holds
  \begin{align*}
    &\int_B w^*(x) (\Delta+c)\varphi(x) dx = \int_{B \cap U} w(x)
    (\Delta+c)\varphi(x) dx \\ &\qquad = \int_{\partial( B \cap U )}
    \big( w(x) \partial_\nu \varphi(x) - \varphi(x) \partial_\nu w(x)
    \big) d\sigma(x) = 0.
  \end{align*}

  Hence we have $(\Delta+c)w^*=0$ in $B$,
  \[
  B = \big( B \cap U \big) \cup \big( B \cap \Gamma \big) \cup \big( B
  \setminus \overline U \big),
  \]
  and $w^*=0$ in $B \setminus \overline U$ which is non-empty since
  Lipschitz domains are by definition on one side of their
  boundary. Holmgren's uniqueness theorem for distributions shows that
  $w^*=0$ in a neighbourhood of $B \cap \Gamma$ \cite{Holmgren,
    Treves}, so $w=0$ in a non-empty open subset of $U$. It is
  real-analytic in $U$, and so must vanish everywhere there by
  connectedness.
\end{proof}

\begin{proof}[Proof of Theorem~\ref{nonConstantThm}]
  By the Rellich's theorem and unique continuation arguments of any of
  \cite{BPS,PSV,EH1,HSV} we have $u=u'$ outside $P\cup P'$, and so
  Proposition~\ref{equalCells} implies $P=P'$. Inside $P$ the total
  waves $u,u'$ satisfy
  \[
  \big(\Delta+k^2(1+\varphi)\big)u = 0, \qquad
  \big(\Delta+k^2(1+\varphi')\big)u' = 0,
  \]
  because $\chi_P = 1$ there. Lemma~\ref{boundary2corner} shows then
  that
  \[
  k^2 (\varphi-\varphi')(x_c) u(x_c) = k^2 (\varphi-\varphi')(x_c)
  u'(x_c) = 0
  \]
  for any vertex $x_c$ of $P$. Since $u(x_c)\neq0$ or $u'(x_c)\neq0$
  we have $\varphi(x_c)=\varphi'(x_c)$.
\end{proof}

\begin{proof}[Proof of Theorem~\ref{cellThm}]
  Write the potentials explicitely as
  \[
  V(x) = \sum_{j=1}^\infty V_j \chi_{\Sigma_j}, \qquad V'(x) =
  \sum_{j=1}^\infty V_j' \chi_{\Sigma_j}
  \]
  for some sequences of constants $V_j,V_j'\in\C$.
  
  By Rellich's theorem and unique continuation the total waves $u,u'$
  satisfying
  \[
  \big(\Delta+k^2(1+V)\big)u = 0, \qquad \big(\Delta+k^2(1+V')\big)u'
  = 0
  \]
  are equal in
  $\operatorname{int}\big(\R^n\setminus\cup_j\overline{\Sigma_j})$
  since $\cup_j \overline{\Sigma_j}$ is simply connected. We shall
  denote this exterior domain by $\Sigma_0$. We will prove the
  following claim by induction: $V_j=V_j'$ and $u=u'$ in $\Sigma_j$
  for $j<m$. The case $m=1$ is true with the interpretation that
  $V_0=V_0'=0$.

  Assume that the induction hypothesis holds for $j<m$. If
  $\Sigma_m=\emptyset$ then $V_m=V_m'=0$ and the induction step is
  trivial. Otherwise let $x_c$ be a vertex of $\Sigma_m$ that's
  connected to infinity by a path that stays a uniform positive
  distance away from $\Sigma_k$ for $k>m$. This means that there is
  $r>0$ such that $B(x_c,r) \setminus \overline{\Sigma_m}$ belongs to
  $\Sigma_0 \cup \ldots \cup \Sigma_{m-1}$. The induction assumption
  implies that $u=u'$ on $B(x_c,r) \setminus
  \overline{\Sigma_m}$. Hence also $u=u'$ and $\partial_\nu
  u=\partial_\nu u'$ on $B(x_c,r) \cap \partial \Sigma_m$. Proposition
  \ref{boundary2corner} and $u(x_c)\neq0$ or $u'(x_c)\neq0$ imply that
  $V_m=V_m'$. Finally, Holmgren's uniqueness theorem of Lemma
  \ref{holmgren} implies that $u=u'$ on $\Sigma_m$. The induction is
  complete.
\end{proof}

\begin{proof}[Proof of Theorem~\ref{nestedThm}]
  Write out the nested structures of $V$ and $V'$ as
  \[
  V = \sum_{j=1}^\infty V_j \chi_{\Sigma_j}, \qquad V' =
  \sum_{j=1}^\infty V_j' \chi_{\Sigma'_j}
  \]
  where $\Sigma_j = D_j \setminus \overline{D_{j+1}}$ for admissible
  cells or empty sets $D_j$ with
  \[
  D_j \Supset D_{j+1}.
  \]
  Similarly for $V'$.

  Let us prove that $D_j=D_j'$, $V_j=V_j'$ when $j<m$ and $u=u'$
  outside $D_m\cup{D_m'}$ by induction on $m$. Interpret
  $D_0=D_0'=\R^n$, $V_0=V_0'=0$. By the Rellich's theorem and unique
  continuation arguments of any of \cite{BPS,PSV,EH1,HSV} we have
  $u=u'$ outside $D_1\cup D_1'$ and so the case $m=1$ is true.

  Let the above claim hold for some fixed $m$. We have $u=u'$ outside
  $D_m\cup D_m'$. Moreover $V_m\neq V_{m-1}=V_{m-1}'\neq V_m'$ so $V$
  and $V'$ have jumps at the vertices of $D_m$ and $D_m'$,
  respectively and are equal outside $D_m\cup{D_m'}$.
  Proposition~\ref{equalCells} implies $D_m=D_m'$ and $V_m=V_m'$. The
  same conclusion holds also when either of them is the empty set
  after which we may stop. By Holmgren's uniqueness theorem (Lemma
  \ref{holmgren}) we have $u=u'$ in $D_m \setminus (D_{m+1} \cup
  D_{m+1}')$ and so anywhere outside $D_{m+1}\cup{D_{m+1}'}$. The
  induction step is complete.
\end{proof}

\section*{Acknowledgement}

The work of H Liu was supported by a startup fund from City University of Hong Kong and the Hong Kong RGC general research funds (projects
12302017, 12301218, 12302919). The authors would like to thank the anonymous referee for the constructive comments and suggestions, which have lead to significant improvements on the presentation and the result of the paper.

\end{document}